\documentclass[12pt,letterpaper,oneside]{amsart}

\usepackage{amsfonts,amsthm,amsmath,amssymb,latexsym}
\usepackage{graphicx,color}
\usepackage[all]{xy}

\usepackage{amssymb}
\usepackage{epstopdf}

\usepackage{geometry} 
\usepackage{graphicx}
\usepackage{amsfonts,amsthm,amsmath,amssymb,latexsym, amscd, euscript}
\usepackage{graphicx,color}
\usepackage[all]{xy}
\usepackage{epsfig}
\usepackage{labelfig}
\usepackage{mathrsfs}

\usepackage{amssymb}
\usepackage{epstopdf}
\begin{document}

\theoremstyle{definition} 

 \newtheorem{definition}{Definition}[section]
 \newtheorem{remark}[definition]{Remark}
 \newtheorem{example}[definition]{Example}

\newtheorem*{notation}{Notation} 

\theoremstyle{plain} 

 \newtheorem{proposition}[definition]{Proposition}
 \newtheorem{theorem}[definition]{Theorem}
 \newtheorem{corollary}[definition]{Corollary}
 \newtheorem{lemma}[definition]{Lemma}

\def\H{{\mathbb H}}
\def\F{{\mathcal F}}
\def\R{{\mathbb R}}
\def\Q{{\mathbb Q}}
\def\Z{{\mathbb Z}}
\def\E{{\mathcal E}}
\def\N{{\mathbb N}}
\def\X{{\mathcal X}}
\def\Y{{\mathcal Y}}
\def\C{{\mathbb C}}
\def\D{{\mathbb D}}
\def\G{{\mathcal G}}
\def\T{{\mathcal T}}
\def\S{{\mathbb S}}
\def\PSL{{\mathrm{PSL}}}
\def\QMCG{{\mathrm{Mod}^{qc}}}

\title{Quasisymmetric maps, shears, lambda lengths and flips}


\date{}

\author{Hugo Parlier and Dragomir \v Sari\' c}

\address{Department of Mathematics, University of Luxembourg, Luxembourg}
\email{hugo.parlier@uni.lu}
\address{Department of Mathematics, Queens College (CUNY) and the CUNY Graduate Center, New York, NY}
\email{dragomir.saric@qc.cuny.edu}

\maketitle

\begin{abstract}
We study quasisymmetric maps, which act on the boundary of the hyperbolic plane, by looking at their action on the Farey triangulation. Our main results identify exactly which quasisymmetric maps correspond to pinched lambda lengths in terms of shearing coordinates and, separately, in terms of (simultaneous) flip distance of the Farey triangulation and its image by the map. These extend and clarify previous results of Penner and Sullivan.
\end{abstract}

\section{Introduction}

The theories of quasiconformal, quasisymmetric and bi-Lipschitz maps of the hyperbolic plane are related to a plethora of topics in Teichm\"uller theory and deformation spaces of surfaces. In particular, the so-called universal Teichm\"uller space, the space of all quasisymmetric maps of the unit circle that fix $1, i$ and $-i$, contains as embedded subspaces, all other Teichm\"uller spaces. Qualifying which maps correspond to quasisymmetric maps is a natural problem.

A map  of the unit circle can be studied by investigating its action on one of the standard (ideal) triangulations  of the hyperbolic plane such as the classical Farey triangulation. Penner and Sullivan \cite{Penner1} did exactly this, and gave a sufficient condition for a map to be quasisymmetric in terms of the lambda lengths of the image of the Farey triangulation (see Section \ref{sec:prelim} for definitions). Their result says that if a triangulation has pinched lambda lengths then it corresponds to a quasisymmetric map. 

We show a number of extensions and related results. We begin by showing that the Penner-Sullivan condition is not exhaustive, that is that there exists triangulations that are the image of a quasisymmetric map which do not admit any decorations with pinched lambda lengths (Example \ref{ex:shear}). We then give a characterization of the subclass of quasisymmetric maps which do satisfy the Penner-Sullivan condition. This characterization involves shearing coordinates:

\begin{theorem}
\label{thm:mainA}
Let $h:\hat{\R}\to\hat{\R}$ be a quasisymmetric map and let $s:\F\to\R$ be its shear function. Then the triangulation $h(\F )$ admits a decoration with bounded lambda lengths if and only if there exists $M>0$ such that for all fans $\F^p=\{ E_j^p\}_{j=-\infty}^{\infty}$ and for all $n,m\in\Z$ with $n<m$ we have
\begin{equation*}
\label{eq:sum_shears}
\Big{|}\sum_{j=n}^ms(E_j^p)\Big{|}\leq M.
\end{equation*}
\end{theorem}

It is interesting to compare this result to a previous result of the second author, which characterizes quasisymmetric maps in terms of shearing coordinates (see Theorem \ref{thm:shears-qs} in Section \ref{sec:prelim}). Note that Theorem \ref{thm:shears-qs} and the above result give an alternate proof of the Penner-Sullivan result. 

Flip distances between triangulations are useful and by now classical tools in the study of mapping class groups and Teichm\"uller spaces. A recent focus point has been simultaneous graphs, where any number of flips are allowed on disjoint quadrilaterals. These are particularly well adapted to infinite type surfaces (and surfaces with infinitely many triangles such as the hyperbolic plane) and have been recently studied by the first author and collaborators in different contexts \cite{Disarlo-Parlier, McLeay-Parlier, Fossas-Parlier}. Our next result is about characterizing Penner-Sullivan quasisymmetric maps in terms of (simultaneous) flip distance. 

\begin{theorem}\label{thm:mainB}
A  homeomorphism $h:\hat{\R}\to\hat{\R}$ that preserves $\hat{\mathbb{Q}}$ is of Penner-Sullivan type if and only if the triangulations $\F$ and $h(\F)$ are finite flip distance apart. 
\end{theorem}

The set $\mathcal{M}(\H )$ of Penner-Sullivan type homeomorphisms of $\hat{\R }$ that preserve $\hat{\Q }$ can be thought of as the universal modular group adapted to the Farey triangulation $\F$. It has the nice property of being transitive on the space of ideal triangulations that are finite flip distance away from $\F$. It is not clear (or obvious) that $\mathcal{M}(\H)$ is a group under composition which is needed if we want to call it the universal modular {\it group}. We establish this fact:

\begin{theorem}
The set $\mathcal{M}(\H)$ is a group under composition. \end{theorem}

By previously known results, $\mathcal{M}(\H)$ contains as subgroups the mapping class groups of all finite-type punctured surfaces and the punctured solenoid \cite{PennerSaric}, \cite{BonnotPennerSaric}.  Our definition of $\mathcal{M}(\H )$ can be compared with Penner's definition \cite{Penner1} of a universal modular group using the space of all triangulations of $\H$. In recent work of Frenkel and Penner \cite{FrenkelPenner}, the universal modular group is defined to be the homeomorphisms which map the Farey triangulation onto triangulations that agree with $\F$ except for many edges. This corresponds to being finite flip distance away, but where you only allow one edge to be flipped at a time. Our group $\mathcal{M}(\H )$ has rather simple combinatorial and analytic definitions, and it contains the group from \cite{FrenkelPenner} as well as Thompson's group (see for instance \cite{Fossas-Nguyen}). 

These results also open other avenues of investigation. For instance, it would be interesting to characterize quasymmetric maps among homeomorphisms that preserve $\hat{\Q}$ in terms of flips on $\F$.\\

\noindent{\bf Acknowlegements.} The first author was supported by the Luxembourg National Research Fund OPEN grant O19/13865598. The second author was supported by the Simons Foundation grant 638572 and the PSC CUNY grant 63477.

\section{Shears, lambda lengths and preliminary results}\label{sec:prelim}

Our base object is the hyperbolic plane $\H$. Computations will often be in the upper half-plane model whereas figures will be represented in the disk model however. The Farey triangulation $\F$ of $\H$ is constructed starting from the ideal hyperbolic triangle $\Delta_0$ with vertices $0$, $1$ and $\infty$. The group generated by the hyperbolic reflections in the sides of $\Delta_0$ has $\PSL_2(\mathbb{Z})$ as an index $2$ subgroup. The orbit of this reflection group (or of $\PSL_2(\mathbb{Z})$) of $\Delta_0$ is a an ideal triangulation of $\H$, which we will refer to as the Farey triangulation. Its set of edges is denoted by $\mathcal{F}$ and its set of vertices is the extended set of rational numbers $\hat{\mathbb{Q}}=\Q\cup\{\infty\}=\Q\cup\{\frac{1}{0}\}$. Note that two vertices $\frac{a}{b}$ and $\frac{c}{d}$ are related by an edge if and only if $|ad-bc|=1$. The Farey triangulation can be visualized in the disk model of the hyperbolic plane as in Figure \ref{fig:farey}.

\begin{figure}[h!tbp]
	\centering
	\includegraphics[width=6cm]{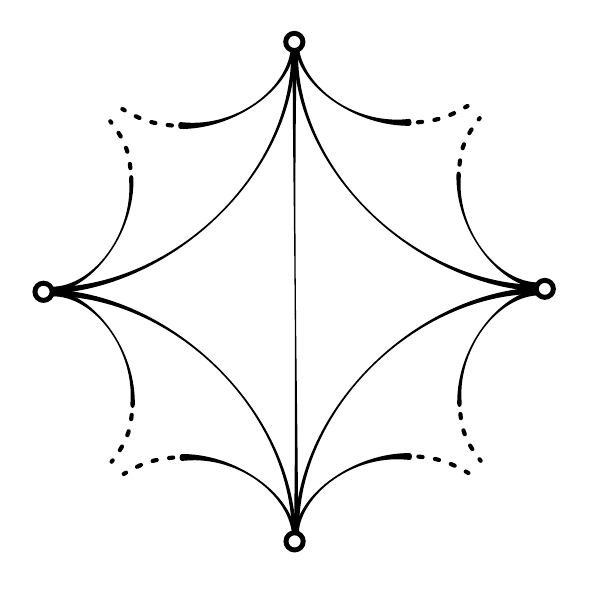}
	\caption{The Farey triangulation}
	\label{fig:farey}
\end{figure}

A horodisk is a disk in $\H$ tangent to a boundary point of $\partial \H$ (so it is of infinite radius), and a horocycle is the boundary of a horodisk. A decoration $H$ of $\mathcal{F}$ is a choice of horocycle tangent to every $q \in \hat{\Q}$. A decoration $H$ is a set, and an element $C\in H$ is a single horocycle. 
The standard horocycle $H_0$ decoration of $\mathcal{F}$ is the decoration that results in the corresponding disks being of maximal density on each triangle of $\H \setminus \mathcal{F}$, meaning that they all look like the ideal triangle in Figure \ref{fig:ideal}.

\begin{figure}[h!tbp]
	\centering
	\includegraphics[width=6cm]{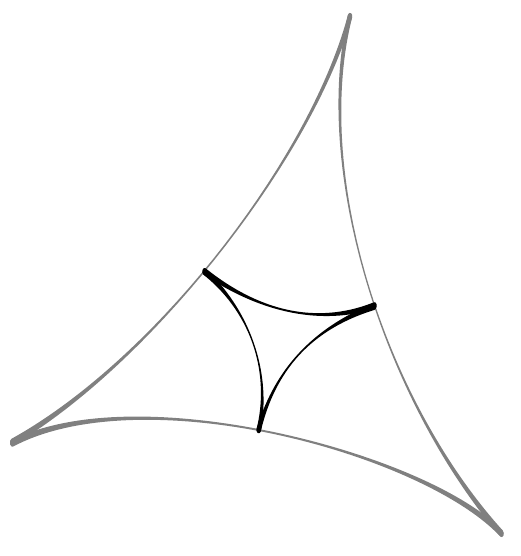}
	\caption{An ideal triangle and tangent horocycles}
	\label{fig:ideal}
\end{figure}

A {\it shear function} is a function $s:\F\to\R$. Geometrically, a shear measures how two ideal triangles are glued together. Take two triangles, $\Delta_1$ and $\Delta_2$, pasted along an edge $E$. We orient $E$ according to the (positive) orientation of $\Delta_1$. Now for each triangle consider the unique circle tangent to all three sides and their tangency points $x_1$ and $x_2$ at $E$ (corresponding to $\Delta_1$ and $\Delta_2$). The shear is the (signed) distance between $x_1$ and $x_2$ following the orientation of $E$. Note that this quantity does not depend on which triangle is chosen to be $\Delta_1$. 

A {\it fan of edges} $\F_p$ with tip $p\in\hat{\Q}$ consists of all edges of $\F$ with $p$ as one of its ideal endpoints. We list $\{ E_j^p\}_{j=-\infty}^{\infty}$ the edges of $\F_p$ such that $E_j^p$ and $E_{j+1}^p$ are adjacent (meaning they share a triangle) and $E_j^p$ comes before $E_{j+1}^p$ for the positive orientation of a horocycle based at $p$.

The shear function $s$ induces a {\it developing map} $h_s:\hat{\Q}\to\hat{\R}$ such that the image triangulation $h_s(\F )$ realizes the shear function. Note that in general $h_s$ does not extend to a homeomorphism of $\hat{\R}$.

Let $k\in\Z$, $n\in\N$ and $\{ E_j^p\}_{j=-\infty}^{\infty}$ be a fan of edges of $\F$ with tip $p$. By pre-composing $h_s$ with an element of $\PSL_2(\Z )$ and post-composing with an element of $\PSL_2(\R )$, we can assume that $h_s$ fixes $-1$, $0$ and $\infty$, that the tip $p$ is at infinity and the edge $E_k^p$ has endpoints $0$ and ${\infty}$. 

The developing map moves an edge $E^p_j$ with endpoints $\infty$ and $j$ to a geodesic with endpoints $\infty$ and $h_s(j)$.
Let $C=\{ z=x+yi:y=1\}$ be a horocycle based at $\infty$. Let $\alpha_j$ be the length of the arc of $C$ between the images by $h_s$ of edges $E_j^p$ and $E_{j+1}^p$.

Then we have
$$
\alpha_j=e^{s(E_k^p)+\cdots +s(E_j^p)}
$$
for $j\geq k$ and
$$
\alpha_j=e^{-s(E_{k-1}^p)-\cdots -s(E_j^p)}
$$
for $j<k-1$ and
$$
\alpha_{k-1}=1.
$$ 

For $k\in Z$ and $n\in\N$, define
$$
s(k,n;p):=\frac{\alpha_k+\alpha_{k+1}+\cdots +\alpha_{k+n-1}}{\alpha_{k-1}+\alpha_{k-2}+\cdots +\alpha_{k-n}}.
$$
The expression $s(k,n;p)$ is the ratio between the sum of the lengths of the first $n$ horocyclic arcs to the right of the geodesic $h_s(E_k^p)$ and the sum of the lengths of the first $n$ horocyclic arcs to the left of $h_s(E_k^p)$. While the lengths of horocyclic arcs $\alpha_j$ depend on the choice of the horocycle, the quotient $s(k,n;p)$ does not.

An orientation preserving homeomorphism $h:\hat{\R}\to\hat{\R}$ is quasisymmetric if there exists $M>0$ such that for any two adjacent arcs I and J with $|I|=|J|$ we have $1/M\leq |h(I)|/|h(J)|\leq M$, where $|I|$ is the length of $I$. The ratios $s(k,n;p)$ enable a precise characterization of quasisymmetric maps via shears:

\begin{theorem}[See \cite{Saric1}, \cite{Saric2}]
\label{thm:shears-qs}
A shear function $$s:\F\to\R$$ is induced by a quasisymmetric map if and only if there exists $M\geq 1$ such that for all fans $\{ E_j^p\}$ of $\F$ and all $k\in\Z$, $n\in\N$
$$
\frac{1}{M}\leq s(k,n;p)\leq M.
$$
\end{theorem}

We now define lambda lengths, which require decorated triangulations. A {\it decorated triangulation} $\tilde{\mathcal{T}}$ of $\H$ is an ideal (locally finite) triangulation $\mathcal{T}$ together with a choice of a horocycle at each vertex. There is a homeomorphism $f$ of $\hat{\R}$ which moves the (vertices of the) Farey triangulation $\F$ onto (those of) $\T$ and the homeomorphism is unique once we decide where a fixed triangle of $\F$ is mapped to. 

A decorated triangulation $\tilde{\mathcal{T}}$ assigns {\it lambda lengths} to the edges of $\F$ under the corresponding homeomorphism as follows (see \cite{Penner}). The lambda length of $E\in\F$ is the quantity
$$
\lambda (E)=\sqrt{e^{\delta (E)}}
$$
where $\delta (E)$ is the signed hyperbolic distance between the intersection of the edge $f(E)\in\mathcal{T}$ with the two horocycles at its endpoints from the decorated triangulation $\tilde{\mathcal{T}}$ (see Figure \ref{fig:lambda}).

\begin{figure}[h!tbp]
\centering
\begin{minipage}[b]{0.3\textwidth} 
	\includegraphics[width=\textwidth]{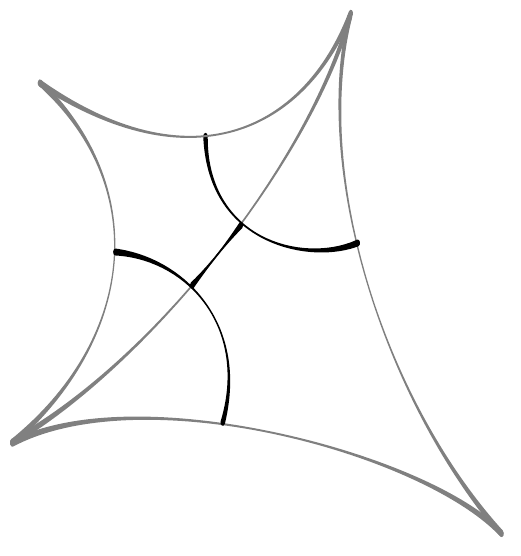}
    	\end{minipage}
~ 
\qquad\qquad
\begin{minipage}[b]{0.3\textwidth} 
	\includegraphics[width=\textwidth]{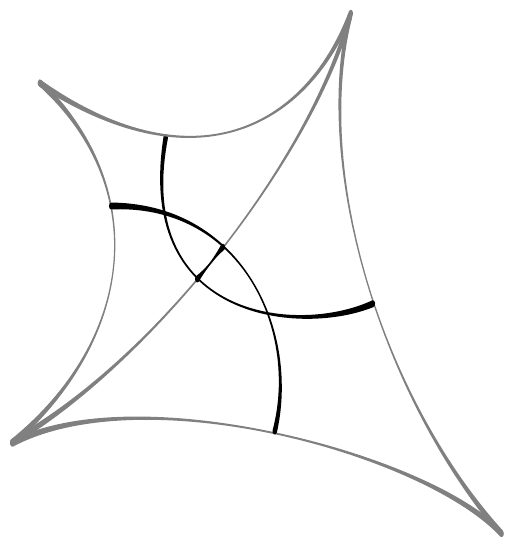}
\end{minipage}
\caption{Positive and negative $\delta(E)$}
\label{fig:lambda}
\end{figure}

Conversely, given an assignment $\lambda :\F\to\R^{+}$ there exists a decorated (possibly degenerate) triangulation $\tilde{\mathcal{T}}$ that realizes $\lambda$ (see \cite{Penner1}, \cite{Penner}). The triangulation may not cover the whole hyperbolic plane and in this case we call it degenerate. Note that this amounts to the developing map not extending to a homeomorphism of $\hat{\R}$. We refer to \cite{Saric1} for a necessary and sufficient condition on the shears or lambda lengths such that the developing map is a homeomorphism. 

The horocycles and edges of $\tilde{\mathcal{T}}$ divide the hyperbolic plane into hexagons and triangles. The hexagons have three sides on the edges and three sides on the horocycles and the triangles have one side on the horocycles and two sides on the edges. There are relations between the horocyclic lengths and the lambda lengths of the sides as follows.

Assume that we have an ideal hyperbolic triangle with vertices $a$, $b$ and $c$ and a fixed choice of corresponding horocycles $C_a,C_b,C_c$ at the vertices. Let $\lambda_{a,b}$ denote the lambda length of the geodesic with endpoints $a$ and $b$. Let $\alpha_a$ be the length of the arc of the horocycle $C_a$ inside the triangle. Then \cite[Lemma 4.9, Page 36]{Penner}
\begin{equation}
\label{eq:horoc-lambda}
\alpha_a=\frac{\lambda_{b,c}}{\lambda_{a,b}\lambda_{a,c}}.
\end{equation}

Assume that we have an ideal hyperbolic quadrilateral with vertices $a$, $b$, $c$ and $d$ in the given order. The Ptolemy relation \cite[Corollary 4.16, Page 41]{Penner} holds
\begin{equation}
\label{eq:ptolemy}
\lambda_{a,c}\lambda_{b,d}=\lambda_{a,b}\lambda_{c,d}+\lambda_{b,c}\lambda_{d,a}.
\end{equation}

\begin{definition}
A lambda length function
$$
\lambda :\F\to\R^{+}
$$
is {\it pinched} if there exists $M>1$ such that
$$
\frac{1}{M}\leq\lambda (E)\leq M
$$
for all $E\in\F$.
\end{definition}

\begin{remark}
We note that having pinched lambda lengths is a global condition which depends on the decoration across all vertices and it is straightforward to find decorations for the identity map that is not pinched. In contrast, the condition in Theorem \ref{thm:shears-qs} is localized to fans and independent of a choice of decoration.
\end{remark}

The following result relates pinched lambda lengths to quasisymmetric maps.
\begin{theorem}[Penner-Sullivan \cite{Penner1}]
\label{thm:lambda-pinched}
Let $\lambda :\F\to\R^{+}$ be pinched. Then the characteristic map is quasisymmetric.
\end{theorem}

\begin{proof}
We show how to deduce this result from Theorem \ref{thm:shears-qs}. Since $\frac{1}{M}\leq \lambda (E)\leq M$ for all $E\in\F$ and by (\ref{eq:horoc-lambda}) we have $\frac{1}{M^3}\leq \alpha (E)\leq M^3$, where $\alpha (E)$ is the length of the horocyclic arc of the decoration with one endpoint on $E$ and to the immediate right of $E$. Then $\frac{1}{M^6}\leq s(k,n;p)\leq M^6$ and $f_\lambda$ is quasisymmetric by Theorem \ref{thm:shears-qs}.
\end{proof}

The following example will show that the condition in the above theorem cannot be an if and only if condition. 

\begin{example}
\label{ex:shear}
We construct a shear function $s:\F\to\R$ corresponding to a quasisymmetric map, but for which any decoration is unpinched.

Let $s:\F\to\R$ be zero on all edges of the Farey triangulation $\F$ that do not belong to the fan $\F^{\infty}$. 
Fix $k\in\N$. Let $E_j$ be the edge with endpoints $j$ and $\infty$. Define 
$$s(E_{-16^j})=\log 2$$ 
and
$$s(E_{16^j})=-\log 2$$
 for $j\in\N\cup\{ 0\}$. We set $s$ to be zero on all other edges of $\F^{\infty}$.
 Let $h_s:\hat{\R}\to\hat{\R}$ be the corresponding homeomorphism that is normalized to be the identity on $[-1,1]$. The map $h_s$ multiples the distances on $[1,16]$ by $\frac{1}{2}$ and the distances on $[16^k,16^{k+1}]$ by $\frac{1}{2^{k+1}}$. 
 
To prove that $h_s$ is quasisymmetric we need to bound the quotient
$$
\frac{h_s(x+t)-h_s(x)}{h_s(x)-h_s(x-t)}.
$$
We assume that $x+t>0$ and $x>0$. The case when $x-t<0$ and $x<0$ can be proved analogously using the symmetry of the shears.

\vskip .2 cm

{\bf Case 1.} Let $16^k\leq x+t<16^{k+1}$, $16^{l-1}\leq x<16^l$ and $l\leq k-1$. Said otherwise, the interval $[x,x+t]$ contains at least one interval of the form $[16^{k-1},16^k]$. Then we have
$$
 -16^{k+1}+2\cdot 16^{l-1}<x-t<-16^k +2\cdot 16^l
$$ 

Note that $x-t$ is negative. By $l\leq k-1$ we conclude
$$
-16^{k+1}\leq x-t\leq -16^{k-1}.
$$

A direct computation gives
$$
h_s(16^k)=\frac{15}{14}8^k-\frac{1}{14}
$$
and by symmetry
$$
h_s(-16^k)=-\frac{15}{14}8^k+\frac{1}{14}.
$$

The above inequalities imply
$$
\frac{15}{14}8^k-\frac{1}{14}\leq h_s(x+t) <\frac{15}{14}8^{k+1}-\frac{1}{14},
$$
$$
\frac{15}{14}8^{l-1}-\frac{1}{14}\leq h_s(x) <\frac{15}{14}8^{l}-\frac{1}{14},
$$
and
$$
-\frac{15}{14}8^{k+1}+\frac{1}{14}\leq h_s(x-t) <-\frac{15}{14}8^{k-1}+\frac{1}{14}.
$$
The above three inequalities imply
$$
\frac{7}{2\cdot 8^2}\leq\frac{h_s(x+t)-h_s(x)}{h_s(x)-h_s(x-t)}\leq 8^2.
$$

\vskip .2 cm

{\bf Case 2.} Assume $x\in [0,1)$. This is similar to Case 1.

\vskip .2 cm

{\bf Case 3.} Assume that $16^{k-1}\leq x\leq 16^k\leq x+t\leq 16^{k+1}$. There are two subcases to consider: either $x-t\geq 16^{k-2}$ or $x-t<16^{k-2}$.

If $x-t\geq 16^{k-2}$ then we have 
$$
\frac{1}{4}\leq\frac{h_s(x+t)-h_s(x)}{h_s(x)-h_s(x-t)}\leq 4
$$
because the interval $[x-t,x+t]$ contains endpoints of at most two geodesics with shears $-\log 2$. 

Now assume that $x-t<16^{k-2}$. Then the interval $[x-t,x]$ contains $[16^{k-2},16^{k-1}]$. This implies
$$
\frac{h_s(x+t)-h_s(x)}{h_s(x)-h_s(x-t)}\leq \frac{\frac{15}{14}(8^{k+1}-8^{k-1})}{\frac{15}{14}(8^{k-1}-8^{k-2})}\leq 8^3.
$$
To find a lower bound, note that $h_s(x+t)-h_s(x)\geq \frac{1}{2^k}t$. We need an upper bound on $h_s(x)-h_s(x-t)$. For $16^l\leq x-t\leq 16^{l+1}<16^{k-1}$ we have
$$
16^{k-1}-16^{l+1}\leq t\leq 16^k-16^l
$$
which implies $\frac{15}{16} 16^{k-1}\leq t \leq 16 \cdot 16^{k-1}$ and
$$
h_s(x)-h_s(x-t)\leq \frac{15}{14}(8^k-8^l)\leq \frac{16^2}{14} \frac{1}{2^k} t.
$$

This gives a lower bound for $\frac{h_s(x+t)-h_s(x)}{h_s(x)-h_s(x-t)}$ when $x-t\geq 1$. 

For $x-t\leq 1$ a similar estimate gives the desired bound.
We have established that $h_s$ is a quasisymmetric map. We now need to show that it does not admit a pinched lambda length decoration.
\end{example}

\begin{proposition}
\label{prop:not-pinched}
Let $s$ be the shear function constructed above (Example \ref{ex:shear}). Then no decoration on $h_s(\F )$ has pinched lambda lengths.
\end{proposition}

\begin{proof}
We choose a horocycle based at $\infty$ to have Euclidean height $1$ and show that any choice of horocycles at other vertices cannot have pinched lambda lengths. (Having pinched lambda lengths is invariant under scaling.) Indeed, the image under $h_s$ of the endpoints of $E_{16^k}$ and $E_{16^k+1}$ are at distance $\frac{1}{2^{k+1}}$. If we want to have pinched lambda lengths on $E_{16^k}$ and $E_{16^k+1}$ we are forced to choose horocycles at the vertices $h_s({16^k})$ and $h_s({16^k+1})$ to be circles with pinched Euclidean radii. This implies that the lambda length on the geodesic with endpoints $h_s({16^k})$ and $h_s({16^k+1})$ is going to zero as $k\to\infty$. Therefore no choice of decorations gives pinched lambda lengths.
\end{proof}

\section{Shears and quasisymmetric maps}

The previous example shows that quasisymmetric maps need not be induced by triangulations with bounded lambda lengths. The quasisymmetric maps with this special property, in light of Theorem \ref{thm:lambda-pinched}, will be said to be of {\it Penner-Sullivan} type. This raises the question of which quasisymmetric maps come from bounded lambda lengths, and in particular whether there is a characterization of this special subclass of quasisymmetric maps as in Theorem \ref{thm:shears-qs}. The following result answers this positively (Theorem \ref{thm:mainA} from the introduction):

\begin{theorem}
\label{thm:dec-bounded-lambda-l}
Let $h:\hat{\R}\to\hat{\R}$ be an (orientation preserving) homeomorphism and let $s:\F\to\R$ be its shear function. Then the triangulation $h(\F )$ admits a decoration with bounded lambda lengths if and only if there exists $M>0$ such that for all fans $\F^p=\{ E_j^p\}_{j=-\infty}^{\infty}$ and for all $n,m\in\Z$ with $n<m$ we have
\begin{equation}
\label{eq:sum_shears}
\Big{|}\sum_{j=n}^ms(E^p_j)\Big{|}\leq M.
\end{equation}
\end{theorem}

\begin{proof}
Assume that $h(\F )$ has a decoration with bounded lambda lengths. In other words, there is a choice of a horocycle at each vertex of $h(\F )$ such that the lengths of the cut offs on each edge of $h(\F )$ are bounded between a negative and positive constant. This implies that the lengths of the horocyclic segments cut off by the adjacent edges of the corresponding vertex are bounded between two positive constants for all the horocycles (see equation (\ref{eq:horoc-lambda}) above).

We fix the standard decoration $H_0$ on $\F$ such that all the horocyclic arcs have length $1$. Let $C_p$ be the horocycle of the standard decoration on $\F$ based at the vertex $p$. Let $C_{h(p)}$ be the horocycle of the decoration on $h(\F )$ based at $h(p)$ that has bounded horocyclic segments. We scale $C_{h(p)}$ to a horocycle $C'_{h(p)}$ that has one horocyclic segment $d$ of length $1$. 

Let $A\in \PSL_2(\mathbb{Z})$, $B\in \PSL_2(\R )$ be such that $A^{-1}(p)=B(h(p))=\infty$, $A^{-1}(C_p)=\{ z\in\H :y=1\}$ and $B(C'_{h(p)})=\{ z\in\H :y=1\}$. Additionally we require that the image under $B$ of the edges $E_1,E_2\in h(\F )$ bounding $d$ map to the geodesics with endpoints $0,\infty$ and $1,\infty$ and that $A$ maps edges of $\F$ with endpoints $0,\infty$ and $1,\infty$ onto the edges $h^{-1}(E_1),h^{-1}(E_2)$. The quasisymmetric map $B\circ h\circ A$ fixes $0$, $1$ and $\infty$ and it is the developing map of the shear function $s\circ A^{-1}:\F\to\R$. The horocyclic segments of $B(C'_{h(p)})$ have lengths pinched between two positive constants. The length of the $n$-th horocyclic segment is given by
$$
e^{\sum_{j=1}^n s\circ A^{-1}(E_j)}
$$
for $n\geq 1$ and by
$$
e^{-\sum_{j=n}^0 s\circ A^{-1}(E_j)}
$$
for $n\leq 0$. Therefore we have an upper bound on $
\Big{|}\sum_{j=n}^ms(E_j)\Big{|}.
$

For the converse, assume that
$$
\Big{|}\sum_{j=n}^ms(E_j)\Big{|}\leq M.
$$
for all $n,m$ and all fans $\F^p$. We need to construct a decoration with bounded lambda lengths. As remarked before, this is equivalent to constructing a decoration with bounded horocyclic segments. 

At a vertex $h(p)$ of $h(\F )$, we choose a horocycle $C_{h(p)}$ that has one segment of length $1$. The lengths of the other segments on $C_{h(p)}$ are either of the form $
e^{\sum_{j=1}^ns(E_j)}$ for $n\geq 1$ or $e^{-\sum_{j=n}^0s(E_j)}$ for $n\leq 0$. Therefore all of the horocyclic segments are bounded, and the lambda lengths are bounded.
\end{proof}

\begin{definition}
Following the above result, an orientation preserving homeomorphism of $\S^1$ will be said to be of {\it Penner-Sullivan type} if the corresponding shear function satisfies inequality (\ref{eq:sum_shears}) for all fans of $\mathcal{F}$.
\end{definition}

\begin{remark}
Defining lambda lengths takes into account the relative positions of the horocycles along the edges of the triangulation. To say that lambda lengths are bounded one could expect that we need to think about these relative positions. However, the second part of the proof above does not require us to think about relative positions of the horocycles because of the formulas in the hexagons. As soon as one of the horocyclic arcs has length $1$, our choice of horocycles is automatically the correct one which simplifies the proof.

\end{remark}

\begin{remark}
Theorem \ref{thm:dec-bounded-lambda-l} above provides yet another proof of the Penner-Sullivan pinched lambda lengths theorem. Indeed, the condition implies that each horocyclic segment length is between $e^{-M}$ and $e^M$. Then the condition of Theorem \ref{thm:shears-qs} is satisfied with the constant $e^{2M}$.
\end{remark}

\section{Simultaneous flips and pinched lambda lengths}

Here we are interested in the relationships between decorated ideal triangulations and flip transformations. Given a triangulation, a {\it flip} is an operation that consists in switching the diagonals of disjoint quadrilaterals (see Figure \ref{fig:flip}). These are sometimes referred to as simultaneous flips as opposed to standard flips were only one edge can be flipped at a time. When a surface is triangulated with an infinite number of triangles, these more general flips allow one to connect a much larger class of triangulations by flip sequences. 
\begin{figure}[h!tbp]
	\centering
	\includegraphics[width=10cm]{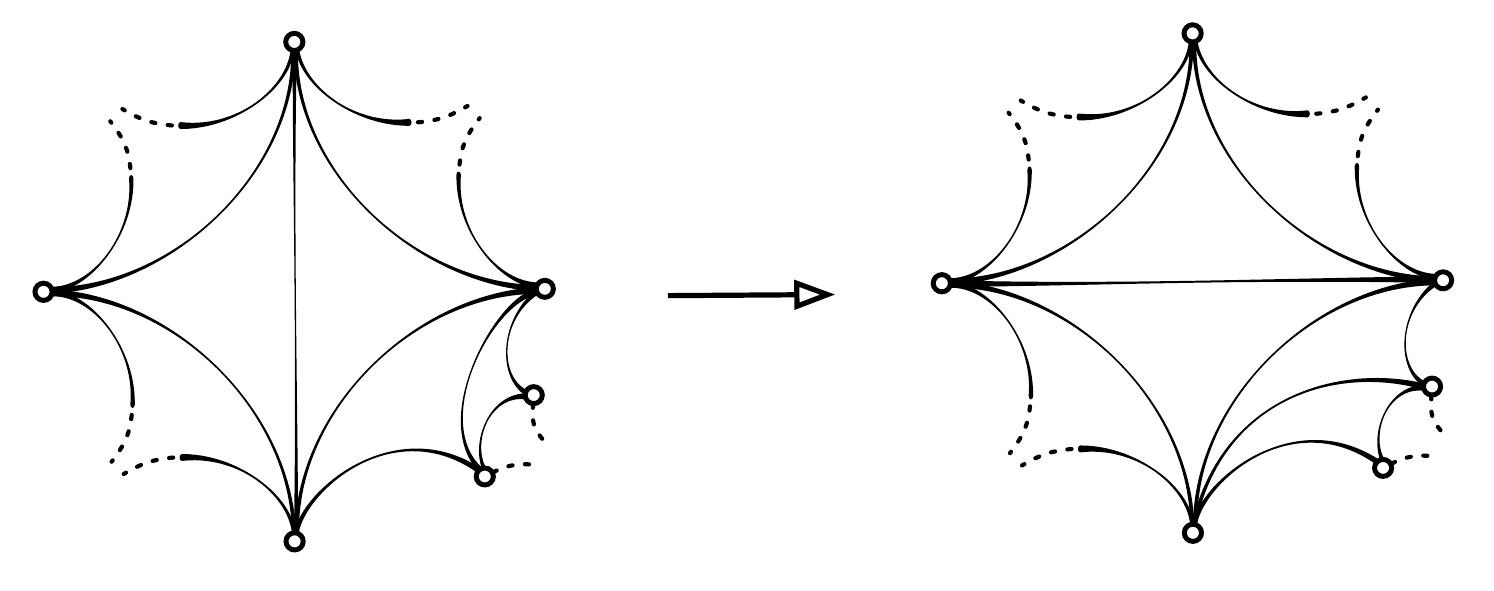}
	\caption{An example of a flip}
	\label{fig:flip}
\end{figure}
Two triangulations are related by a sequence of flips if and only if there exists a universal upper bound on the intersection between the arcs of the triangulations and the other triangulation (see \cite{Fossas-Parlier}). 

\subsection{Pinched lambda lengths under flips}
Let $\tilde{\mathcal{T}}$ be a decorated ideal hyperbolic triangulation of $\H$ with pinched lambda lengths. Let $\mathcal{D}$ be a flip on disjoint quadrilaterals of $\mathcal{T}$.

\begin{proposition}
\label{prop:flip-pinched}
Let $\tilde{\mathcal{T}}$ be a decorated ideal hyperbolic triangulation of $\H$ with pinched lambda lengths. Then the image  $\mathcal{D}(\tilde{\mathcal{T}})$ under a flip $\mathcal{D}$ has pinched lambda lengths. In particular, the developing map $h:\F\to\mathcal{D}(\mathcal{T})$ is of Penner-Sullivan type and thus quasisymmetric.
\end{proposition}

\begin{proof}
Assume that the diagonal $ac$ is flipped to the diagonal $bd$ inside the rectangle with vertices $a$, $b$, $c$ and $d$ made by the union of two triangles of $\mathcal{T}$. 
The Ptolemy equation (\ref{eq:ptolemy}) gives
$$
\lambda_{b,d}=\frac{\lambda_{a,b}\lambda_{c,d}+\lambda_{b,c}\lambda_{d,a}}{\lambda_{a,c}}.
$$
It follows that
$$
\frac{2}{M^3}\leq\lambda_{b,d}\leq 2M^3.
$$
Thus $\mathcal{D}(\tilde{\mathcal{T}})$ has pinched lambda lengths.
\end{proof}
From this, the following is immediate. 
\begin{corollary}
\label{cor:flips-Farey}
The developing map of any triangulation obtained by finitely many flips starting from the Farey triangulation is quasisymmetric.
\end{corollary}

\begin{proof}
Choose a standard decoration on the Farey triangulation such that all lambda lengths are equal to $2$. By Proposition \ref{prop:flip-pinched} the resulting decorated triangulation will be pinched and the result now follows from Theorem \ref{thm:lambda-pinched}.
\end{proof}

\begin{proposition} 
\label{prop:flip-qs}
Let $\mathcal{T}=h(\F )$ for a quasisymmetric map $h$ and $\mathcal{D}$ a flip on $\T$. Then the characteristic map of $\mathcal{D}(\mathcal{T})$ is quasisymmetric.
\end{proposition}
Before passing to the proof, we remark that in general the triangulation $\mathcal{T}$ does not support a decoration with pinched lambda lengths. Therefore the proof of Proposition \ref{prop:flip-pinched} does not work in general.

\begin{proof}
If the flip $\mathcal{D}$ does not preserve a single triangle of $\mathcal{T}$, then we form two flips $\mathcal{D}_1$ on $\mathcal{T}$ and $\mathcal{D}_2$ on $\mathcal{D}_1(\mathcal{T})$ such that $\mathcal{D}_2\circ\mathcal{D}_1=\mathcal{D}$ and both $\mathcal{D}_1$ and $\mathcal{D}_2$ leave at least one triangle fixed. It is enough to prove that the characteristic map $h_1$ for $\mathcal{D}_1(\mathcal{T})$ is quasisymmetric because in that case $\mathcal{D}_1(\mathcal{T})$ is the image of $\F$ under quasisymmetric  $h_1$ and the same argument applies to the flip $\mathcal{D}_2$ on $\mathcal{D}_1(\mathcal{T})$. 

Let $\Delta_1$ be a complementary triangle of the triangulation $\mathcal{T}$ that is not changed by $\mathcal{D}_1$. 
Let $\Delta_1^*$ be the complementary triangle of $\F$ such that $h(\Delta^*_1)=\Delta_1$. 
We define a flip $\mathcal{D}_1^*$ on $\F$ by $h^{-1}\circ\mathcal{D}_1\circ h$. By definition $\mathcal{D}_1^*(\Delta_1^*)=\Delta_1^*$. Let $h^*_1$ be the characteristic map from $\F$ to $\mathcal{D}_1^*(\F )$ that fixes the edges of $\Delta_1^*$. By the definition of $\mathcal{D}_1^*$ we have that $h_1^*(\F )=h^{-1}\circ f_1\circ h(\F )$. Since both $h_1^*$ and $h^{-1}\circ h_1\circ h$ are characteristic maps from $\F$ to $\mathcal{D}_1^*(\F )$ and they both fix the edges of $\Delta^*_1$ we have 
$$
h_1=h\circ h_1^*\circ h^{-1}. 
$$

The map $h_1$ is quasisymmetric because $h$ is quasisymmetric by the assumption and $h_1^*$ is quasisymmetric by Proposition \ref{prop:flip-pinched}.
\end{proof}











\subsection{Intersection and lambda lengths}

The goal of this section is to prove a type of converse to Proposition \ref{prop:flip-pinched}.

\begin{theorem}\label{thm:finiteflip}
Let $\T$ be a triangulation with pinched lambda lengths with respect to the standard horocycle decoration $H_0$ of $\H$. Then $\T$ is finite flip distance from $\F$, the standard Farey triangulation. 
\end{theorem}

To prove it, we use a few preliminary results. The first one states that an arc of $\T$ never travels too deeply into a horocycle unless it ends at its point of tangency on $\hat{\mathbb{R}}$. More precisely:

\begin{lemma} There exists a constant $K'$ such that any arc $a\in \T$ satisfies 
$$\ell(a \cap D_C) \leq K'$$
for any horodisk $D_C$ with horocycle boundary $C\in H_0$ that is not tangent to one of the endpoints of $a$.
\end{lemma}

\begin{proof}
This follows immediately from the pinched lambda lengths condition. Indeed, if there is no upper bound on this length, then the lambda lengths are arbitrarily large.
\end{proof}

Observe that this implies a universal upper bound on the distance between the intersection of an arc $a$ with a horocycle decoration $H_0$, and the horocycle. It is convenient to think of the depth of an arc here: an arc is said to have depth $\delta_a$ if for each $C\in H_0$, $a\cap C$ is at most Hausdorff distance $\delta_a$ from $\partial C$. 

If we let $\delta_0:= \sup_{a\in \T}\{\delta_a\}$ then we can replace $H_0$ with a new horocycle decoration $H_r$ obtained by retracting $H_0$ by $\delta_0+1$ (see Figure \ref{fig:retract}). 

\begin{figure}[h!tbp]
	\centering
	\includegraphics[width=4cm]{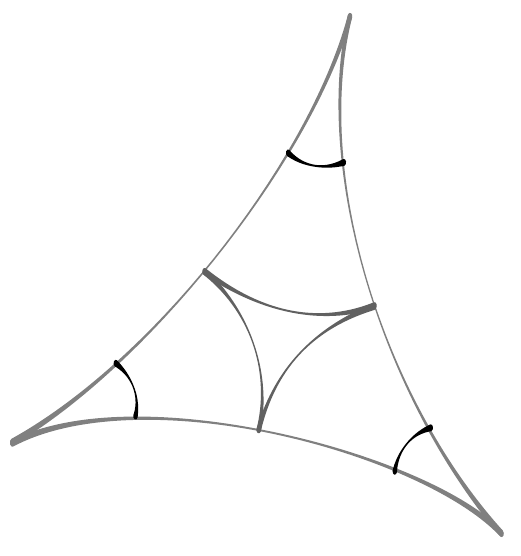}
	\caption{The standard decoration is replace by a retracted decoration}
	\label{fig:retract}
\end{figure}

The result of this is that $\T$ and $H_r$ are, by construction, completely disjoint. 

Our second preliminary result involves a uniform upper bound on the intersection between arcs of $\T$ and $\F$ and vice versa. 

\begin{lemma}
There exists $K''$ such that for any $a\in \T$ we have 
$$
i(a,\F) \leq K''.
$$
\end{lemma}
\begin{proof}
Cutting $\F$ along the horocycles of $H_r$ (by which we mean removing the horodisks), results in a collection of hexagons. These hexagons pasted together is the ``large" connected component of $\H\setminus H_r$. We now consider $\T$ on this truncated hyperbolic plane $\H_{H_r}$ and we continue to denote it $\T$. An arc $a$ of $\T$ on $\H_{H_r}$ is finite length, and leaves from one of the horocycle boundaries, crosses a finite number of arcs of $\F$, and then returns to another horocycle boundary in its other endpoint. Each intersection point with an arc of $\F$ creates length: that is there is a lower bound of its length in terms of intersection with $\F$ that goes to infinity as the intersection goes to infinity (see Figure \ref{fig:interlength}). 

\begin{figure}[h!tbp]
	\centering
	\includegraphics[width=6cm]{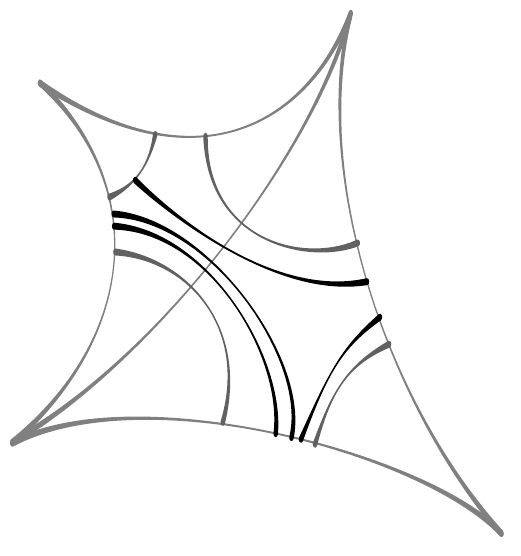}
	\caption{Each time an arc of $\T$ crosses an arc of $\F$ this forces it to have a certain length}
	\label{fig:interlength}
\end{figure}

Note that this is exactly why we retracted by $\delta_0 + 1$. As we have a universal upper bound on the length of $a$ restricted to $\H_{H_r}$, this proves a uniform upper bound on $i(a,\F)$ as desired.
\end{proof}

The proof of Theorem \ref{thm:finiteflip} now follows directly from the main result of \cite{Fossas-Parlier} which says that the finite intersection property above implies finite flip distance. From Theorem \ref{thm:finiteflip} and Proposition \ref{prop:flip-qs} we now have, as an immediate corollary, a characterization of being of Penner-Sullivan type in terms of flips (Theorem \ref{thm:mainB} from the introduction):

\begin{corollary}\label{cor:ps-flip}
A homeomorphism $h:\hat{\R}\to\hat{\R}$ such that $h(\hat{\mathbb{Q}})=\hat{\mathbb{Q}}$ is of Penner-Sullivan type if and only if the triangulations $\F$ and $h(\F)$ are finite flip distance apart. 
\end{corollary}

The above corollary characterizes when a homeomorphism $h$ that preserves $\hat{\mathbb{Q}}$ is of Penner-Sullivan type in terms of the intersections of $\F$ and $h(\F )$. We pose the following problem.

\vskip .2 cm

\noindent {\it Open problem}: Characterize when a homeomorphism  $h:\hat{\R}\to\hat{\R}$ such that $h(\hat{\mathbb{Q}})=\hat{\mathbb{Q}}$ is quasisymmetric in terms of the intersection properties of $\F$ and $h(\F )$.

\section{The universal modular group}

The above characterization of quasisymmetric maps encourages a more systematic study of triangulations of $\H$ with vertex set $\hat{\mathbb{Q}}$. The quasiconformal mapping class group $\QMCG(\H )$ of the hyperbolic plane consists of all quasisymmetric maps of $\hat{\mathbb{R}}$ without requiring to fix three points (see \cite{GardinerLakic} and \cite{MarkovicSaric}). Somewhat surprisingly, the group $\QMCG(\H )$, which acts by isometries on the universal Teichm\"uller space $T(\H)$, contains the whole space and more. In order to make the theory more in line with finite surfaces (which was Penner's original intention) we introduce a countable group acting on the universal Teichm\"uller space (the normalized quasisymmetric maps) and the space of Penner-Sullivan maps arising by the flip construction.

\begin{definition}
An {\it allowable triangulation} of $\H$ is a locally finite ideal triangulation whose set of vertices is $\hat{\mathbb{Q}}$. (By locally finite we mean that every compact region of the hyperbolic plane intersects a finite number of edges of the triangulation.)
\end{definition}

\begin{definition}
Two allowable triangulations $\T_1$ and $\T_2$ have the {\it finite intersection property} if 
$$
\sup_{\alpha_1\in\T_1}i(\alpha_1, \T_2)<\infty
$$
and
$$
\sup_{\alpha_2\in\T_2}i(\alpha_2, \T_1)<\infty .
$$
\end{definition}

As we have just seen, the maps of Penner-Sullivan type that preserve the rationals are of the following form:
\begin{definition}
The set $\mathcal{M}(\mathbb{H})$ is the set of all homeomorphisms $h: \hat{\mathbb{R}}\to\hat{\mathbb{R}}$ preserving $\hat{\mathbb{Q}}$ such that $\F$ and $h(\F )$ have the finite intersection property. 
\end{definition}

\begin{lemma}
\label{lem:transitive-fip}
The finite intersection property on allowable triangulations is transitive.
\end{lemma}

\begin{proof}
Assume that $\T_1$ and $\T_2$, and $\T_2$ and $\T_3$ have the finite intersection property. We will show that $\T_1$ and $\T_3$ have the finite intersection property. Let $n=\sup_{\alpha_1\in\T_1} i(\alpha _1,\T_2)$. Then any arc $\alpha_1\in\T_1$ is covered by at most $n+2$ complementary triangles of $\T_2$. Let $m=\sup_{\alpha_2\in\T_2} i(\alpha_2 ,\T_3)$. Then any arc $\alpha_2\in\T_2$ is covered by at most $m+2$ complementary triangles of $\T_3$. Then any arc $\alpha_1\in\T_1$ is covered by at most $3(n+2)(m+2)$ complementary triangles to $\T_3$ and therefore it is intersected by at most $9(n+2)(m+2)$ arcs of $\T_3$. Therefore $\sup_{\alpha_1\in\T_1}i(\alpha_1,\T_3)\leq 9(n+2)(m+2)$ and by symmetry we obtain $\sup_{\alpha_3\in\T_3}i(\alpha_3,\T_1)<\infty$. Thus $\T_1$ and $\T_3$ have the finite intersection property.
\end{proof}

\begin{proposition}
\label{prop:M(H)-group}
The set $\mathcal{M}(\mathbb{H})$ is a group under composition.
\end{proposition}

\begin{proof}
Assume $h\in \mathcal{M}(\mathbb{H})$. Then $h(\F )$ and $\F$ have the finite intersection property. This property remains invariant under homeomorphisms of $\hat{\mathbb{R}}$ that preserve $\hat{\mathbb{Q}}$. Therefore $h^{-1}(h(\F ))=\F$ and $h^{-1}(\F )$ have the finite intersection property and $h^{-1}\in \mathcal{M}(\mathbb{H})$.

Let $h_1,h_2\in \mathcal{M}(\mathbb{H})$. Then $\F$ and $h_1(\F )$, and $\F$ and $h_2(\F )$ have the finite intersection property. The finite intersection property is preserved under homeomorphisms of $\hat{\mathbb{R}}$ that setwise fix $\hat{\mathbb{Q}}$. Then $h_2(\F )$ and $h_2(h_1(\F ))$ have the finite intersection property. By Lemma \ref{lem:transitive-fip}, $\F$ and $h_2\circ h_1(\F )$ have the finite intersection property. Therefore $h_2\circ h_1\in \mathcal{M}(\mathbb{H})$.
\end{proof}

We end with the observation that the group $\mathcal{M}(\H)$ contains as interesting proper subgroups (lifts of) the mapping class (semi-)groups of finite-type punctured surfaces (see \cite{Penner1} and \cite{Penner}) and the baseleaf preserving mapping class group of the punctured solenoid (see \cite{PennerSaric} and \cite{BonnotPennerSaric}).

\end{document}